\documentclass[a4paper,12pt]{amsart}  
\usepackage{amsmath,amssymb,amsthm}     
\usepackage[utf8]{inputenc}
\usepackage{enumitem}
\usepackage{color}
\theoremstyle{plain}      
 
\newtheorem{thm}{Theorem}[section]     
\newtheorem{theorem}[thm]{Theorem}

\newtheorem{lemma}[thm]{Lemma}     
\newtheorem{prop}[thm]{Proposition}

\theoremstyle{definition}      
     
\newtheorem{definition}[thm]{Definition}  
 
\newtheorem{remark}[thm]{Remark}

\DeclareMathAlphabet{\doba}{U}{msb}{m}{n}

\gdef\mN{\doba{N}}

\gdef\mR{\doba{R}}

\def\CO{\mathrm{CO}}

\def\di{{\rm d}}

\let\<\langle 
\let\>\rangle 
\let\ti\tilde

\newcommand{\definedas}{\mathrel{\raise.095ex\hbox{\rm :}\mkern-5.2mu=}}

\title{}

\author{Andrei Moroianu, Mihaela Pilca}

\address{Andrei Moroianu \\ Université Paris-Saclay, CNRS,  Laboratoire de mathématiques d'Orsay, 91405, Orsay, France}
\email{andrei.moroianu@math.cnrs.fr}

\address{Mihaela Pilca\\Fakult\"at f\"ur Mathematik\\
Universit\"at Regensburg\\Universit\"atsstr. 31 
D-93040 Regensburg, Germany}
\email{mihaela.pilca@mathematik.uni-regensburg.de}

\subjclass[2010]{}

\keywords{}

\begin{document}   
	
	\title{Conformal vector fields on lcK manifolds}

	\begin{abstract}
		We show that any conformal vector field on a compact  lcK manifold is Killing with respect to the Gauduchon metric. Furthermore, we prove that any conformal vector field on a compact lcK manifold whose K\"ahler cover is neither flat, nor hyperk\"ahler, is holomorphic.
	\end{abstract}
	\maketitle
	
	\section{Introduction}
	
	It is well known that on a compact Kähler manifold every conformal vector field is Killing \cite[§90]{l}, and every Killing vector field is holomorphic. The aim of this paper is to extend these two results to compact locally conformally Kähler (lcK) manifolds. 
	
	Recall that a (compact) lcK manifold \cite{v} is a compact complex manifold $(M,J)$ together with a conformal class $c$ of Riemannian metrics such that in the neighbourhood of each point of $M$ there exists a Kähler metric in $c$ compatible with $J$. Equivalently, $(M,J,c)$ is lcK if the universal cover $\widetilde M$ of $M$ carries a Kähler metric $g_K$ in the induced conformal class $\tilde c$ compatible with the induced complex structure $\tilde J$. The simply connected Kähler manifold $(\widetilde M,\ti J, g_K)$ will henceforth be referred to as the Kähler cover of $(M,J,c)$.
	
	The interest of this notion is that many complex manifolds which for topological reasons do not carry Kähler metrics (like most complex surfaces with odd first Betti number \cite{f}, Hopf manifolds $S^1\times S^{2n-1}$, some OT manifolds \cite{ot}, etc.) have lcK structures instead.
	
	Every compact lcK manifold $(M,J,c)$ carries a distinguished metric $g_0\in c$, uniquely defined up to a positive constant, called the Gauduchon metric \cite{g}. Given a conformal vector field $\xi$ on $(M,c)$, one cannot reasonably hope that it preserves any metric in the conformal class, simply because if $g\in c$ is preserved by $\xi$, then for any smooth function $f$ non-constant along the flow of $\xi$, the conformally equivalent metric $\ti g:=e^{2f}g$ is no longer preserved by $\xi$. 
	What one can hope, however, is to show that $\xi$ preserves the Gauduchon metric $g_0$. Note that if $\xi$ were also holomorphic, this would be almost tautological. Indeed, since $g_0$ is defined up to a constant by $c$ and $J$, the flow of $\xi$ would be homothetic with respect to $g_0$, and on a compact Riemannian manifold every homothetic vector field is Killing. 

Our first result (Theorem \ref{thmkilling} below) says that this is indeed the case: {\em every conformal vector field on a compact lcK manifold preserves the Gauduchon metric}. This result was conjectured and proved under some more restrictive assumptions in \cite{mo}.

We then move to the next natural question: is every conformal vector field on a compact lcK manifold holomorphic? It turns out that in this generality the answer is negative. Indeed, one can easily construct lcK metrics with non-holomorphic Killing vector fields on Hopf manifolds $S^1\times S^{2n-1}$ and on products of $S^1$ with 3-Sasakian manifolds  (see \cite[Remark 2.4 (ii)]{mo}). 

However, these are basically the only possible counterexamples: our second main result (Theorem \ref{thmholomorph} below) states that {\em if $(M,J,c)$ is not a Hopf manifold or locally conformally hyperkähler} (that is, if the Kähler metric on the universal cover is not hyperkähler or flat), then {\em every conformal vector field is holomorphic}. 

Unlike the analogous result on Kähler manifolds, which is a simple consequence of Cartan's formula (see {\em e.g.} \cite[Prop. 15.5]{kg}), this extension to lcK geometry is highly non-trivial, and is based on a recent result by M.~Kourganoff \cite[Theorem 1.5.]{k} which describes compact lcK manifolds whose Kähler cover is reducible and non-flat. 

Let us now explain in more detail the strategy of the proofs. We start by showing (in Prop. \ref{propkaehler}) that on a Kähler manifold (not necessarily compact), the divergence of any conformal vector field is harmonic. Note that in the compact case, this already implies Lichnerowicz' result mentioned above.
We then consider a conformal vector field $\xi$ on a compact lcK manifold $(M,J,c)$ and apply this result to the lift $\ti \xi$ of $\xi$ to the Kähler cover $(\widetilde M,\ti J, g_K)$ of $(M,J,c)$. Using the theory of Weyl structures and the existence of Gauduchon metrics, we show in Prop. \ref{propDharmonic} that $\ti\xi$ has constant divergence on $\widetilde M$ with respect to $g_K$. We then interpret this condition in terms of the Gauduchon metric on $M$ and conclude by an integration argument, using the compactness of $M$.

The proof of Theorem \ref{thmholomorph} goes roughly as follows. If $\xi$ is a conformal vector field on $(M,J,c)$, then its lift $\tilde \xi$ is not only conformal, but even homothetic on the Kähler cover $(\widetilde M,\ti J, g_K)$, thanks to Theorem \ref{thmkilling}. In particular $\ti\xi$ is affine, {\em i.e.} preserves the Levi-Civita connection of $g_K$. An easy argument  \cite[Lemma 2.1]{mo} shows that if $g_K$ is irreducible and not hyperkähler, then $\tilde\xi$ is holomorphic. 

In the case where $g_K$ is non-flat but has reducible holonomy, we make use of a deep result by M. Kourganoff, stating that $(\widetilde M, g_K)$ is a Riemannian product with a non-trivial flat factor $\mR^q$. Using the fact that $\pi_1(M)$ acts on $\widetilde M$ cocompactly and properly discontinuously by similarities of the metric $g_K$, preserving the homothetic vector field $\ti\xi$, we then show in Proposition \ref{propN} that the component of $\ti\xi$ on $\mR^q$ vanishes. This is the core of the argument and uses the explicit form of conformal vector fields on flat spaces. 

The end of the proof uses a result by K.P. Tod \cite[Prop. 2.2]{tod} involving Einstein-Weyl structures, and the irreducibility of non-flat cone metrics over complete manifolds proved by S. Gallot \cite[Prop. 3.1]{gallot}.

	{\bf Acknowledgments.} This work was supported by the Procope Project No. 57445459 (Germany) /  42513ZJ (France).
	
	\section{Preliminaries}
	In this preliminary section we briefly recall the main definitions and collect a few known basic results that will be needed throughout the paper.
	
	Let $M$ be a smooth $n$-dimensional manifold. For every real number $r$, the weight bundle $L^r$ is the real line bundle associated to the frame bundle of $M$ with respect to the representation $|\mathrm{det}|^{\frac r n }$. Two Riemannian metrics $g,\ \widetilde g$ on $M$ are said to be conformally equivalent if there exists a function $f$ such that $\widetilde g=e^{2f}g$. A conformal structure on $M$ is an equivalence class of Riemannian metrics with respect to this equivalence relation. 
	
	If $c$ is a conformal structure and $g\in c$ is a Riemannian metric, its volume element $\mathrm{vol}_g$ is a section of $L^{-n}$. The volume element of a conformally equivalent metric  $\widetilde g=e^{2f}g$ is $\mathrm{vol}_{\widetilde g}=e^{-nf}\mathrm{vol}_g$, thus showing that $(\mathrm{vol}_g)^{\frac2n}\otimes g$ is a section of $L^{-2}\otimes \mathrm{Sym}^2(\mathrm{T}^*M)$ which does not depend on the choice of $g$. We will sometimes identify $c$ with this section.
	
	Let $(M,c)$ be an $n$-dimensional conformal manifold. A vector field $\xi$ on $M$  is called conformal if its flow preserves the conformal class $c$, \emph{i.e.} for any metric $g\in c$, its Lie derivative with respect to $\xi$ is proportional to~$g$: $\mathcal{L}_\xi g=\lambda g$, for some function $\lambda\in\mathcal{C}^\infty(M)$. 
	
	We recall that on a given Riemannian manifold $(M,g)$, the divergence of a vector field is the trace of the endomorphism $\nabla^g X$ of the tangent bundle: $\mathrm{div}^g X:=\mathrm{tr}(\nabla^g X)$, where $\nabla^g$ is the Levi-Civita connection of $g$. The divergence of a vector field is equal to the opposite of the codifferential of its dual $1$-form $X^\flat:=g(X, \cdot)$, \emph{i.e.} $\mathrm{div}^{g} X=-\delta^g (X^\flat)$, where the codifferential $\delta^g$ is the formal adjoint of the exterior differential $\di$ and is expressed in terms of a local $g$-orthonormal basis $\{e_i\}_{i}$ as follows: $\delta^g\alpha=-\displaystyle\sum_{i=1}^n e_i\lrcorner \nabla^g_{e_i} \alpha$, for all forms $\alpha$ on $M$. In the sequel we will drop the metric $g$ in the notation each time the metric is clear from the given context.
	
		Taking traces in the defining equality of a conformal vector field, $\mathcal{L}_\xi g=\lambda g$, shows that  necessarily $\lambda=-\frac{2}{n} \delta^g\xi^\flat$, for any metric $g\in c$. Thus, if $\xi$ is a conformal vector field on $(M,c)$, then
	\begin{equation}\label{eqliemetric}
	\mathcal{L}_{\xi} g=-\frac{2}{n}(\delta^g \xi^\flat) g, \quad \forall g\in c.
	\end{equation}
	In particular, a conformal vector field $\xi$ on $(M,c)$ is Killing with respect to some metric $g\in c$ if and only if $\delta^g\xi^\flat=0$.
	
	The condition that a vector field $\xi$ is conformal is also equivalent to the fact that the covariant derivative of $\xi^\flat$ with respect to any metric $g\in c$ has no trace-free symmetric component, \emph{i.e.}:
	$$\nabla^{g}_X\xi^\flat=\frac{1}{2}X\lrcorner \di\xi^\flat-\frac{1}{n}(\delta^g \xi^\flat)X^\flat, \quad \forall X\in \mathrm{T}M.$$

	\begin{definition}
	A {\it Weyl structure} on a conformal manifold $(M,c)$ is a torsion-free linear connection $D$ which preserves the conformal class $c$. If $D$ has reducible holonomy, then $(M,c,D)$ is called {\it Weyl-reducible}.
	
	\end{definition}
The condition that $D$ preserves the conformal class $c$ means that for each metric $g\in c$, there exists a unique $1$-form $\theta^g\in\Omega^1(M)$, called the {\it Lee form} of $D$ with respect to $g$, such that 
\begin{equation}\label{dg}Dg=-2\theta^g\otimes g.
\end{equation}
The Weyl connection $D$ is then related to $\nabla^g$ by
	\begin{equation}\label{weylstr}
	D_X=\nabla^g_X+\theta^g(X)\mathrm{Id} + \theta^g\wedge X, \quad \forall X\in  \mathrm{T}M,
	\end{equation}
	where $\theta^g\wedge X$ is the skew-symmetric endomorphism of $\mathrm{T}M$ defined by $$(\theta^g\wedge X)(Y):=\theta^g(Y) X-g(X,Y)(\theta^g)^\sharp.$$
	
	A Weyl connection $D$ is called {\it closed} if it is locally the Levi-Civita connection of a (local) metric in $c$ and is called {\it exact} if it is the Levi-Civita connection of a globally defined metric in~$c$. Equivalently, $D$ is closed (resp. exact) if its Lee form with respect to one (and hence to any) metric in $c$ is  closed (resp. exact).
	Note that in the particular case when the Weyl structure is exact, $D=\nabla^{\widetilde g}$ with $\widetilde g=e^{2f}g$, the Lee form $\theta^g$ of $D$ with respect to $g$ is given by $\theta^g=\di f$. This immediately follows from \eqref{dg}, since $Dg=\nabla^{\widetilde g}(e^{-2f}\widetilde g)=-2\di f\otimes(e^{-2f}{\widetilde g})=-2\di f \otimes g$.

	If the manifold $M$ is compact of dimension greater than $2$, then for every Weyl connection~$D$ on $(M,c)$ there exists a unique (up to homothety) metric $g_0\in c$, called the {\it Gauduchon metric} of $D$, such that its associated Lee form $\theta_0$ is co-closed with respect to $g_0$, cf.~\cite{g}.
	
	%Let us remark that with the above normalization of $\theta$, we have in the special case when $D$ is the standard Weyl structure of an lcK manifold $(M, J, [g])$, that the Lee form of $g$ equals $-2\theta$, \emph{i.e.} $\di\Omega=-2\theta\wedge \Omega$, where $\Omega:=g(J\cdot, \cdot)$ is the fundamental $2$-form.
	
	The natural extension of \eqref{weylstr} to the bundle of  exterior $k$-forms reads:
	\begin{equation}\label{weylstrform}
	D_X\alpha=\nabla^g_X\alpha-k\theta^g(X)\alpha + X \wedge (\theta^{g})^{\sharp}\lrcorner \alpha-\theta^g \wedge (X\lrcorner \alpha), \quad \forall X\in \mathrm{T}M, \ \forall  \alpha\in\Omega^k(M).
	\end{equation}

	The codifferential $\delta^D:\Omega^k(M)\to L^{-2}\otimes \Omega^{k-1}(M)$ associated to a Weyl structure $D$ on $(M,c)$  is defined as follows:
	$$\delta^D\alpha=-\mathrm{tr}_{c}(D\alpha),$$
	where $\mathrm{tr}_{c}$ denotes the conformal trace with respect to $c$. More precisely, if $c= \ell^2\otimes g$, then $\delta^D$ is related to $\delta^g$ by the following formula
	\begin{equation}\label{deltaweyl1}
	\delta^D\alpha=\ell^{-2}(\delta^g\alpha-(n-2k)\theta^\sharp\lrcorner\alpha),
	\end{equation}
which directly follows by applying~\eqref{weylstrform} to any $k$-form $\alpha$ and a local $g$-orthonormal basis $\{e_i\}_i$:
	\begin{equation*}
	\begin{split}
	-\sum_{i=1}^ne_i\lrcorner D_{e_i}\alpha&=-\sum_{i=1}^ne_i\lrcorner \nabla^{g}_{e_i}\alpha+ k\theta^\sharp\lrcorner\alpha-n\theta^\sharp\lrcorner\alpha+(k-1)\theta^\sharp\lrcorner\alpha+\theta^\sharp\lrcorner\alpha\\
	&=\delta^g\alpha-(n-2k)\theta^\sharp\lrcorner\alpha.
	\end{split}
	\end{equation*}
	An exterior form $\alpha$ satisfying $\delta^D\alpha=0$ is called {\it $D$-coclosed}. According to \eqref{deltaweyl1}, $\alpha$ is $D$-coclosed if and only if for any metric $g\in c$, the codifferential of $\alpha$ verifies $\delta^g\alpha-(n-2k)\theta^\sharp\lrcorner\alpha=0$, where $\theta$ is the Lee form of $D$ with respect to $g$.
	
	The Weyl Laplacian $\Delta^D:\mathcal{C}^\infty(M)\to \mathcal{C}^\infty(L^{-2})$ is defined by $$\Delta^D\varphi:=\delta^D\di\varphi=-\mathrm{tr}_c(D\di\varphi),\quad  \forall \varphi\in\mathcal{C}^\infty(M).$$
	 For every metric $g\in c$ written as $g=\ell^{-2}\otimes c$, \eqref{deltaweyl1} applied to the $1$-form $\alpha=\di\varphi$ yields
	\begin{equation}\label{laplaceweyl} \Delta^D\varphi=\ell^{-2}(\Delta^g\varphi+(2-n)g(\theta,\di\varphi)), \quad \forall \varphi\in\mathcal{C}^\infty(M).
	\end{equation}
	A function $\varphi\in\mathcal{C}^\infty(M)$ satisfying $\Delta^D\varphi=0$ is called {\it $D$-harmonic}.
	
	\begin{lemma}\label{l22}
		On a Riemannian manifold $(M,g)$ the commutator between the Lie derivative and the codifferential acting on $1$-forms satisfies the following equation:
		\begin{equation}\label{commLdelta}
		[\delta,\mathcal{L}_X] Y^\flat=\delta((\mathcal{L}_Xg)(Y))-g(Y^\flat, \di(\delta X^\flat)), \quad \forall X, Y\in\Gamma( \mathrm{T}M).
		\end{equation}
	\end{lemma}
	
	\begin{proof} We can assume that $M$ is oriented, up to passing to a double cover.
		If $\mathrm{vol}_g$ denotes the volume form of $g$, then using the well known formula $\mathcal{L}_X\mathrm{vol}_g=-\delta X^\flat \,\mathrm{vol}_g$ (see for instance \cite[Appendix 6]{kn}) we compute for all vector fields $X$ and $Y$:
		\begin{equation*}
		\begin{split}
		\delta( [X,Y]^\flat) \,\mathrm{vol}_g&=-\mathcal{L}_{[X, Y]}\mathrm{vol}_g=-[\mathcal{L}_X, \mathcal{L}_Y]\mathrm{vol}_g=\mathcal{L}_X(\delta Y^\flat \,\mathrm{vol}_g)-\mathcal{L}_Y(\delta X^\flat \,\mathrm{vol}_g)\\
		&=X(\delta Y^\flat)\mathrm{vol}_g-Y(\delta X^\flat)\mathrm{vol}_g,
		\end{split}
		\end{equation*}
		and thus 
		\begin{equation}\label{eqcomm}
		\delta( [X,Y]^\flat)=X(\delta Y^\flat)-Y(\delta X^\flat)=\mathcal{L}_X (\delta Y^\flat)-g(Y^\flat, \di(\delta X^\flat)).
		\end{equation}
		We now compute the commutator as follows:
		\begin{equation*}
		\begin{split}
		[\delta,\mathcal{L}_X](Y^\flat)&=\delta(\mathcal{L}_X Y^\flat)-\mathcal{L}_X (\delta Y^\flat)=\delta([X,Y]^\flat)+\delta((\mathcal{L}_Xg)(Y))-\mathcal{L}_X (\delta Y^\flat)\\
		&\overset{\eqref{eqcomm}}{=}\delta((\mathcal{L}_Xg)(Y))-g(Y^\flat, \di(\delta X^\flat)).
		\end{split}
		\end{equation*}
	\end{proof}

Recall that a complex manifold $(M,J, c)$ endowed with a conformal structure $c$ is called locally conformally K\"ahler (lcK) if around each point of $M$, every metric $g\in c$ can be conformally rescaled to a K\"ahler metric. Equivalently, $(M,J, c)$ is lcK if every $g\in c$ is Hermitian with respect to $J$ and the fundamental 2-form $\Omega:=g(J\cdot, \cdot)$ satisfies $\di\Omega=-2\theta\wedge\Omega$ for some closed 1-form $\theta$ called the Lee form of $(M,c,J)$ with respect to $g$.

If $(M,J, c)$ is lcK, the universal cover $\pi\colon \widetilde M\to M$, endowed with the induced complex structure $\tilde J$ and conformal structure $\tilde c$, admits a K\"ahler metric in $\tilde c$ with respect to which $\pi_1(M)$ acts by holomorphic homotheties. 

Indeed, if $g\in c$ is any metric on $(M,J)$ with Lee form~$\theta$, then the pull-back $\tilde\theta$ is exact on $\widetilde M$, \emph{i.e.} $\tilde\theta=\di\varphi$, for some function $\varphi\in\widetilde M$, and the metric $g_K:=e^{2\varphi}\tilde g$ is K\"ahler. Moreover, $\pi_1(M)$ acts on $(\widetilde M,\tilde J,g_K)$ by holomorphic homotheties. Hence, the Levi-Civita connection of $g_K$ projects to a closed, non-exact, Weyl connection $D$ on $M$, the so-called {\it standard Weyl connection} of the lcK manifold $(M,J, c)$, whose Lee form in the sense of \eqref{dg} is exactly $\theta$.
	
	\section{Conformal vector fields on K\"ahler manifolds}
	
	In this section we show that the divergence of a conformal vector field on a (not necessarily compact) K\"ahler manifold  is a harmonic function with respect to the K\"ahler metric.

	Let us first recall some well known results in K\"ahler geometry, whose proofs can be found for instance in \cite{kg}. Let $(M, J, g, \Omega)$ be an $n$-dimensional K\"ahler  manifold. 
	
	In the sequel $\{e_i\}_{i}$ denotes a local orthonormal basis with respect to the metric $g$. Then the K\"ahler $2$-form can be written as $\Omega=\displaystyle\frac{1}{2}\sum_{i=1}^n e_i\wedge Je_i$, where here and in the sequel we identify vectors and 1-forms using the metric $g$. We denote by $L$ the wedge product with $\Omega$: 
	$$L\colon \Omega^k(M)\to\Omega^{k+2}(M), \quad L(\alpha)=\Omega\wedge \alpha.$$
	
	The natural extension of $J$ acting as a derivation on forms is given by
	$$J\colon\Omega^k(M)\to\Omega^k(M), \quad J(\alpha)=\sum_{i=1}^n Je_i\wedge e_i\lrcorner\alpha.$$
	
	The twisted differential $\di^c$ is defined as follows:
	$$\di^c\colon\Omega^k(M)\to\Omega^{k+1}(M), \quad \di^c(\alpha)=\sum_{i=1}^n Je_i\wedge \nabla_{e_i}\alpha,$$
	and its formal adjoint is 
	$$\delta^c\colon\Omega^{k+1}(M)\to\Omega^{k}(M), \quad \delta^c(\alpha)=-*\di^c *=-\sum_{i=1}^n Je_i\lrcorner \nabla_{e_i}\alpha,$$
	where $\nabla$ denotes the Levi-Civita connexion of $g$.	The twisted Laplace operator is then defined by $\Delta^c:=\di^c\delta^c+\delta^c\di^c$.
	
	On a K\"ahler manifold, the following relations hold (for a proof see for instance \cite[§14]{kg}):
	\begin{equation}\label{Jd}
	[J,\di]=\di^c, \quad [J,\delta^c]=-\delta,
	\end{equation}
	\begin{equation}\label{deltacd}
	\delta^c\di+\di\delta^c=\delta^c\delta+\delta\delta^c=\di^c\delta+\delta\di^c=0,
	\end{equation}
	\begin{equation}\label{deltacL}
	[L, \delta^c]=-\di,
	\end{equation}
	\begin{equation}\label{laplacekaehler}
	\Delta^c=\Delta.
	\end{equation}
	
	After these preliminaries we can now prove the announced result:
	
	\begin{prop}\label{propkaehler}
		Let $(M, J, g, \Omega)$  be a (not necessarily compact) K\"ahler manifold of dimension $n>2$. If  $\eta$ is a conformal Killing $1$-form on $(M, g)$, then its codifferential $\delta\eta$ is a $g$-harmonic function.
	\end{prop}
	
	\begin{proof}
		Let $\eta$ be a conformal Killing $1$-form on $(M,g)$, \emph{i.e.} the dual $1$-form of a conformal vector field on $M$. The covariant derivative of $\eta$ in the direction of any vector field $X$ is given as follows (see \cite{uwe}):
		$$\nabla_X \eta=\frac{1}{2}X\lrcorner \di\eta-\frac{1}{n}(\delta \eta)X.$$
		We thus compute using the above commutator relations:
		\begin{equation*}
		\begin{split}
		\di^c \eta&=\sum_{i=1}^n Je_i \wedge \nabla_{e_i}\eta=\frac{1}{2}\sum_{i=1}^n Je_i\wedge e_i\lrcorner \di\eta-\frac{1}{n}(\delta \eta)\sum_{i=1}^n Je_i \wedge e_i\\
		&=\frac{1}{2}J\di\eta+\frac{2}{n}(\delta\eta)\Omega\overset{\eqref{Jd}}{=}\frac{1}{2}\di J\eta+\frac{1}{2}\di^c\eta+\frac{2}{n}L(\delta\eta),
		\end{split}
		\end{equation*}
		hence $\di^c\eta=\di J\eta+\frac{4}{n}L(\delta\eta)$. Applying $\delta^c$ to this equality yields
		\begin{equation*}
		\begin{split}
		\delta^c\di^c\eta&=\delta^c\di J\eta+\frac{4}{n}\delta^c L(\delta\eta)\overset{\eqref{deltacd}, \eqref{deltacL}}{=}-\di \delta^cJ\eta+\frac{4}{n}L\delta^c (\delta\eta)+\frac{4}{n}\di(\delta\eta)\overset{\eqref{Jd}}{=}-\di \delta\eta+\frac{4}{n}\di(\delta\eta),
		\end{split}
		\end{equation*}
		and thus $\delta^c\di^c\eta+\frac{n-4}{n}\di(\delta\eta)=0.$ Applying now $\delta$ to this equality yields
		$$0=\delta\delta^c\di^c\eta+\frac{n-4}{n}\delta\di(\delta\eta)\overset{\eqref{deltacd}}{=}\delta^c\di^c(\delta\eta)+\frac{n-4}{n}\delta\di(\delta\eta)=\Delta^c(\delta\eta)+\frac{n-4}{n}\Delta(\delta\eta)\overset{\eqref{laplacekaehler}}{=}\frac{2n-4}{n}\Delta(\delta\eta).$$
		Since $n>2$, it follows that $\Delta(\delta\eta)=0$, so $\delta\eta$ is $g$-harmonic.
	\end{proof}	
	
	\begin{remark}
		In terms of vector fields, Proposition~\ref{propkaehler} can be reformulated as follows:\\  The divergence of a conformal vector field on a K\"ahler manifold of real dimension greater than $2$ is a harmonic function with respect to the K\"ahler metric.
	\end{remark}
	
	\begin{remark}	
		If the manifold $M$ is moreover assumed to be compact, a direct consequence of Proposition \ref{propkaehler} is the well-known result of A. Lichnerowicz \cite[§90]{l} stating that a conformal vector field on a compact K\"ahler manifold of real dimension greater than 2 is necessarily Killing with respect to the K\"ahler metric.
	\end{remark}
	
	\section{Weyl-harmonic functions}
	
	In this section we prove that harmonic functions with respect to a Weyl structure on a compact conformal manifold are necessarily constant.
	
	\begin{prop}\label{propDharmonic}
		Let $(M,c)$ be a compact conformal manifold of dimension $n>2$ endowed with a Weyl structure~$D$. Then any $D$-harmonic function on $M$ is constant.
	\end{prop}	
	
	\begin{proof}
		We consider the Gauduchon metric $g_0\in c$, which is (up to homothety) the unique metric in $c$ whose associated $1$-form $\theta_0$ is $g_0$-coclosed. If $c=\ell_0^2\otimes g_0$, then \eqref{laplaceweyl} yields:
		$$\Delta^{D}\varphi=\ell_0^{-2}(\Delta^{g_0}\varphi+(2-n)g_0(\di\varphi, \theta_0)), \quad \text{ for all }\varphi\in \mathcal{C}^\infty(M).$$
		Thus a function $\varphi\in \mathcal{C}^\infty(M)$ is $D$-harmonic if and only if
		$$\Delta^{g_0}\varphi=(n-2)g_0(\di\varphi, \theta_0).$$
		Multiplying this equality with $\varphi$ and integrating over the compact manifold $M$ yields
		$$\int_M |\di\varphi|^2_{g_0}\mathrm{vol}_0= \int_M \varphi \Delta^{g_0}\varphi \mathrm{vol}_0=\frac{n-2}{2}\int_M g_0(\di\varphi^2, \theta_0)\mathrm{vol}_0=\frac{n-2}{2}\int_M g_0(\varphi^2, \delta_0\theta_0)\mathrm{vol}_0=0,$$
		which implies that $\di\varphi=0$. Thus $\varphi$ is constant, since $M$ is compact. 
	\end{proof}	
	
	\begin{remark}
		Let $(M^n,c)$ be a conformal manifold of dimension $n>2$ endowed with a Weyl structure~$D$. To each vector field $\xi$ can be associated the following function
		\begin{equation}\label{deffunction}
		f_{\xi}:=\mathrm{div}^{\nabla^g}\xi+n\theta(\xi),
		\end{equation}
		where $g\in c$ and $\theta$ is its associated $1$-form. Then $f_\xi$ is independent of the choice of the metric $g\in c$, as shown by a direct computation. Namely, if $\widetilde g=e^{2f}g$, then $\widetilde \theta=\theta-\di f$. Taking a local orthonormal basis $\{e_i\}_i$ with respect to $g$, then $\{\tilde e_i:=e^{-f}e_i\}_i$ is a local orthonormal basis with respect to $\tilde g$ and we obtain using \cite[Thm. 1.159 a)]{besse}:
		\begin{equation*}
		\begin{split}
		\mathrm{div}^{\nabla^{\widetilde g}}\xi+n\widetilde\theta(\xi)&=\sum_{i=1}^n \widetilde g(\tilde e_i, \nabla^{\widetilde g}_{\tilde e_i}\xi)+n(\theta-\di f)(\xi)\\
		&=\sum_{i=1}^n g(e_i, \nabla^{\widetilde g}_{e_i}\xi)+n(\theta-\di f)(\xi)\\
		&=\sum_{i=1}^n g(e_i, \nabla^{g}_{e_i}\xi +\di f(e_i)\xi+\di f(\xi)e_i-g(e_i, \xi)\mathrm{grad}f)+n(\theta-\di f)(\xi)\\
		&=\mathrm{div}^{\nabla^{g}}\xi+\di f (\xi)+n\di f (\xi)-\di f (\xi)+n(\theta-\di f)(\xi)\\
		&=\mathrm{div}^{\nabla^{g}}\xi+n\theta(\xi).
		\end{split}
		\end{equation*}
	\end{remark}
	
	\section{Conformal vector fields on lcK manifolds}
	
	We are now ready to prove the counterpart in lcK geometry of the above mentioned result of A.~Lichnerowicz for compact K\"ahler manifolds. More precisely, we show the following: 
	
	\begin{theorem}\label{thmkilling}
		Let $(M,J,c)$ be a compact lcK manifold. Then every conformal vector field on $(M, c)$ is Killing with respect to the Gauduchon metric and the induced vector field on the universal cover is homothetic with respect to the Kähler metric.
	\end{theorem}
	
	\begin{proof}
		Let $\xi$ be a conformal vector field on $(M,c)$ and let $\eta_0:=g_0(\xi, \cdot)$ be its dual 1-form with respect to the Gauduchon metric $g_0$. Then $\xi$ is Killing with respect to $g_0$ if and only if $\delta_0\eta_0=0$. 
		
		We consider the universal cover $\pi\colon \widetilde M\to M$ endowed with the pull-back $(\tilde J,\tilde g_0,\tilde \theta_0)$ of the lcK structure $(J, g_0, \theta_0)$, where $\theta_0$ is the Lee form defined by $\di\Omega_0=-2\theta_0\wedge\Omega_0$. If $\varphi\in\mathcal{C}^\infty(\widetilde M)$ is a primitive of $\tilde\theta_0$, \emph{i.e.} $\tilde\theta_0=\di\varphi$, then the metric $g_K:=e^{2\varphi}\tilde g_0$ is K\"ahler.
		
		We denote by $\ti\xi$ the vector field induced by $\xi$ on $\widetilde M$, \emph{i.e.} $\pi_*\ti\xi=\xi$. Then $\ti\xi$  is a conformal vector field with respect to the conformal class $[\ti g_0]=[g_K]$, and thus its dual $1$-form $\eta_K:=g_K(\ti\xi, \cdot)$ is a conformal Killing $1$-form on the K\"ahler manifold $(\widetilde M, \ti J, g_K)$. The pull-back $\ti\eta_0$ of $\eta_0$ is related to $\eta_K$ by $\ti\eta_0=e^{-2\varphi}\eta_K$.
		
		We claim that $\delta_{g_K}\eta_K=-\pi^*f_\xi$, where $f_\xi$ is the function associated to the vector field~$\xi$, as defined by~\eqref{deffunction}. Indeed, we compute using the formula for the conformal change of the codifferential \cite[Thm. 1.159 i)]{besse}:
		\begin{equation*}
		\begin{split}
		\delta_{g_K}\eta_K&=e^{-2\varphi}(\delta_{\ti g_0}\eta_K-(n-2)\ti g_0(\eta_K, \di \varphi))=e^{-2\varphi}\delta_{\ti g_0}(e^{2\varphi}\ti\eta_0)-(n-2)\ti g_0(\ti\eta_0, \di \varphi)\\
		&=\delta_{\ti g_0}\ti\eta_0-2\ti g_0(\ti\eta_0, \di \varphi)-(n-2)\ti g_0(\ti\eta_0, \di \varphi)=\delta_{\ti g_0}\ti\eta_0-n\ti g_0(\ti\eta_0, \ti\theta_0)\\
		&=\pi^*(\delta_{g_0}\eta_0-n g_0(\eta_0, \theta_0))=\pi^*(-\mathrm{div}^{\nabla^{g_0}}\xi-n  \theta_0(\xi))=-\pi^*(f_\xi).
		\end{split}
		\end{equation*}	
		Since by Proposition~\ref{propkaehler}, the function $\pi^*(f_\xi)=-\delta_{g_K}\eta_K$ is $g_K$-harmonic, it follows that $f_\xi$ is $D$-harmonic, where $D$ is the standard Weyl structure of the lcK structure $(M,J,c)$. Applying Proposition~\ref{propDharmonic} we obtain that $f_\xi$ is constant, so \mbox{$f_{\xi}=C\in\mR$}. On the other hand, using the Gauduchon metric $g_0$ with its associated $1$-form~$\theta_0$, we  express $f_\xi$ as follows:
		$$C=f_\xi=\mathrm{div}^{\nabla^{g_0}}\xi+n\theta_0(\xi)=-\delta_0\eta_0+n\theta_0(\xi),$$
		hence 
		\begin{equation}\label{cst}\theta_0(\xi)=\frac{C}{n}+\frac{1}{n}\delta_0\eta_0.\end{equation}
		 By Cartan's formula we further compute: 
		$$\mathcal{L}_{\xi}\theta_0=\di(\xi\lrcorner \theta_0)+\xi\lrcorner \di\theta_0=\di(\theta_0(\xi))=\frac{1}{n}\di(\delta_0\eta_0).$$
		Applying now the codifferential $\delta_0$ to this equality and using Lemma \ref{l22}, we obtain:
		\begin{equation*}
		\begin{split}
		\frac{1}{n}\Delta_0(\delta_0\eta_0)&=\frac{1}{n}\delta_0\di(\delta_0\eta_0)=\delta_0\left(\mathcal{L}_{\xi}\theta_0\right)\\
		&\overset{\eqref{commLdelta}}{=}\mathcal{L}_{\xi}(\delta_0\theta_0)+\delta_0((\mathcal{L}_\xi g)(\theta_0^\sharp))-g_0(\theta_0, \di(\delta_0\eta_0))\\
		&\overset{\eqref{eqliemetric}}{=}-\frac{2}{n}\delta_0((\delta_0\eta_0)\theta_0)-g_0(\theta_0, \di(\delta_0\eta_0))\\
		&=-\frac{2}{n}(\delta_0\eta_0)\delta_0\theta_0+\frac{2}{n}g_0(\theta_0, \di(\delta_0\eta_0))-g_0(\theta_0, \di(\delta_0\eta_0))\\
		&=\frac{2-n}{n}g_0(\theta_0, \di(\delta_0\eta_0)),
		\end{split}
		\end{equation*}
		since $\delta_0\theta_0=0$ by the definition of the Gauduchon metric. Thus, we obtain:
		\begin{equation}\label{eq1}
		\Delta_0(\delta_0\eta_0)=(2-n)g_0(\theta_0, \di(\delta_0\eta_0)).
		\end{equation}
		Multiplying \eqref{eq1} with the function $\delta_0\eta_0$ and integrating over the compact manifold $M$ yields:
		$$\int_M (\delta_0\eta_0) \Delta_0(\delta_0\eta_0) \mathrm{vol}_{g_0}=\frac{2-n}2\int_M g_0(\theta_0, \di((\delta_0\eta_0)^2))\mathrm{vol}_{g_0},$$
     	so we obtain:
		$$\int_M  |\di(\delta_0\eta_0)|_0^2 \mathrm{vol}_{g_0}=\frac{2-n}2\int_M g_0(\delta_0\theta_0, (\delta_0\eta_0)^2)\mathrm{vol}_{g_0}=0,$$
		showing that $\di(\delta_0\eta_0)=0$. As $M$ is compact, it follows that the function $\delta_0\eta_0$ is constant and hence  vanishes, since  $\displaystyle\int_M  \delta_0\eta_0\, \mathrm{vol}_{g_0}=0$. Thus $\delta_0\eta_0=0$,  so $\xi$ is a Killing vector field with respect to the Gauduchon metric $g_0$.
		
		In order to prove the last statement, we remark that 
		$$\mathcal{L}_{\ti\xi}g_K=\mathcal{L}_{\ti\xi}(e^{2\varphi}\ti g_0)=2\ti\theta_0(\ti\xi)g_K=2\widetilde{\theta_0(\xi)}g_K,$$
		and $\theta_0(\xi)$ is constant by \eqref{cst} and the fact that  $\delta_0\eta_0=0$.
	\end{proof}

\section{Holomorphicity of conformal vector fields}

In this section we prove that on a compact lcK manifold, whose K\"ahler cover is neither flat nor hyperk\"ahler, every conformal vector field  is holomorphic. We first show a more general result about homothetic invariant vector fields on Riemannian products endowed with a cocompact and properly discontinuous action of a group of similarities.

Let us introduce the notation needed in the sequel. We denote the group of similarities of  a Riemannian manifold $(M,g)$ by
$$\mathrm{Sim}(M,g):=\{\phi\colon M\to M\, |\, \phi \text{ is a diffeomorphism and } \phi^*g=\lambda g, \text{for some } \lambda\in\mR_{>0}  \}.$$
 A vector field $\xi$ on $(M,g)$ is called homothetic if its flow consists of similarities. 

%Hence, if one of the components $\xi_1$ or $\xi_2$ vanishes, then $\lambda=0$, so $\xi$ is a Killing vector field with respect to the metric $g_1+g_2$.

\begin{prop}\label{propN}
	Let  $(N,g_N)\times (\mR^q,g_{\text{flat}})$ be the Riemannian product of a non-flat incomplete Riemannian manifold $(N, g_N)$ with irreducible holonomy and the Euclidean space $(\mR^q,g_{\text{flat}})$. We assume that there exists a  subgroup $\Gamma\subset \mathrm{Sim}(N\times \mR^q,g_N+g_{\text{flat}})$ which acts cocompactly and properly discontinuously on $N\times \mR^q$. Then every $\Gamma$-invariant homothetic vector field on $(N\times \mR^q,g_N+g_{\text{flat}})$ is tangent to $N$ and constant in the direction of $\mR^q$.
\end{prop}	

\begin{proof}
Let us denote in this proof by $\pi\colon N\times \mR^q\to (N\times \mR^q)/\Gamma$ the projection given by the action of~$\Gamma$. Let $X$ be a $\Gamma$-invariant homothetic vector field on $(N\times \mR^q,g_N+g_{\text{flat}})$. We write $X=X_1+X_2$, with $X_1\in \mathrm{T}N$ and $X_2\in \mathrm{T}\mR^q$. The flow $(\psi_t)_t$ of $X$ preserves the decomposition $\mathrm{T}N\oplus\mR^q$, because any similarity preserves the flat factor of the de Rham decomposition. Thus, the following inclusions hold: $\psi_{t*}(\mathrm{T}N)\subset \mathrm{T}N$ and $\psi_{t*}(\mR^q)\subset \mR^q$, so $[X, \mathrm{T}N]\subset \mathrm{T}N$ and $[X, \mR^q]\subset \mR^q$, which further imply that $\nabla_{\mathrm{T}N} X\subset \mathrm{T}N$ and $\nabla_{\mR^q} X\subset \mR^q$, where $\nabla$ denotes the Levi-Civita connection of $g_N+g_{\text{flat}}$. Hence, 
\begin{equation}\label{constant}\nabla_{\mathrm{T}N} X_2=0\quad\mathrm{and}\quad\nabla_{\mR^q} X_1=0,
\end{equation} 
showing that $X_1$ and $X_2$ are conformal vector fields on the factors and are constant in the direction of the other factor. Clearly $X_1$ and $X_2$ are $\Gamma$-invariant. We need to show that $X_2=0$.

The conformal vector field $X_2$ on the Euclidean space $\mR^q$ is given as follows at each $p\in\mR^q$: $(X_2)_p=Cp+v$, where $C=\lambda I_q+ A$, for some skew-symmetric matrix $A\in M(q, \mR)$,  $\lambda\in\mR$ and $v\in\mR^q$, and $I_q\in M(q, \mR)$ denotes the identity matrix. 

We claim that by applying a translation in $\mR^q$ one can assume that $v\in\mathrm{Ker}(C)$. In order to prove this we distinguish the following two cases:\\
{\bf Case 1.} If $\lambda\neq 0$, then $C$ is invertible and $(X_2)_p=C(p+C^{-1}v)$, so choosing the origin of the flat factor $\mR^q$ at $-C^{-1}v$, we may assume that $v=0$.\\
{\bf Case 2.} If $\lambda=0$, then $C=A$ is a skew-symmetric matrix. Considering the orthogonal splitting $\mR^q=\mathrm{Im}(C) \oplus \mathrm{Ker}(C)$, we decompose correspondingly $v=Cv_1+v_2$, with $v_2\in  \mathrm{Ker}(C)$. Thus $(X_2)_p=C(p+v_1)+v_2$, so choosing the origin of the flat factor $\mR^q$ at $-v_1$, we may assume that $v\in\mathrm{Ker}(C)$.

The flow of $X_2$ is given as follows: $\varphi_t(p)=e^{tC}\left(p+\displaystyle\int_0^t e^{-sC}v \di s\right)$. Since $v\in\mathrm{Ker}(C)$, this further simplifies to $\varphi_t(p)=e^{tC}p+tv$.

 Now, for every element $\gamma\in \Gamma$, we write $\gamma(x,y)=(\gamma'(x), \gamma''(y))$, for all $(x,y)\in N\times \mR^q$, with $\gamma''(y)=B_{\gamma} y+w_{\gamma}$, $B_{\gamma}\in \CO(q)$ and $w_{\gamma}\in\mR^q$. Since $X_2$ is $\Gamma$-invariant, it follows that its flow $(\varphi_t)_t$ commutes with $\gamma''$, that is:
$$B_{\gamma}(e^{tC}p+tv)+w_{\gamma}=e^{tC}(B_{\gamma} p+w_{\gamma})+tv,\quad \forall p\in \mR^q,\; \forall t\in\mR,$$
or, equivalently:
\begin{equation*}
\begin{cases}
[B_{\gamma}, e^{tC}]=0\\
t(B_{\gamma}-I_q)v=(I_q-e^{-tC})w_{\gamma}.
\end{cases}
\end{equation*}
Differentiating at $t=0$ yields for all $\gamma\in\Gamma$:
\begin{equation}\label{eqv}
\begin{cases}
[B_{\gamma}, C]=0\\
(B_{\gamma}-I_q)v=Cw_{\gamma}.
\end{cases}
\end{equation}

We claim that $Cw_{\gamma}=0$, for all $\gamma\in\Gamma$. We show this  separately for the two cases introduced above. In the first case, if $\lambda\neq 0$, then we have already seen that we may assume $v=0$. Hence, \eqref{eqv} directly implies that $Cw_{\gamma}=0$ for all $\gamma$. In the second case, if $\lambda=0$, then $C$ is a skew-symmetric matrix and we have shown that one may assume $v\in\mathrm{Ker}(C)$. Therefore, since $B_{\gamma}$ and $C$ commute, it follows that the left-hand side of the second equality in \eqref{eqv} also belongs to the kernel of $C$, $(B_{\gamma}-I_q)v\in\mathrm{Ker}(C)$. The right-hand side belongs to $\mathrm{Im}(C)$ and since $\mathrm{Im}(C)\perp \mathrm{Ker}(C)$, because $C$ is skew-symmetric, we conclude that both sides of this equality vanish, so, in particular, $C w_{\gamma}=0$.

Since $Cw_{\gamma}=0$ and there exists at least a strict homothety $B_{\gamma}$, it follows from the second equality in \eqref{eqv} that $v=0$. Thus, $(X_2)_p=Cp$, for all $p\in\mR^q$.

Assume, for a contradiction, that $X_2\neq 0$, \emph{i.e.} $C\neq 0$. Let us now fix some $p\in\mathrm{Im}(C)\setminus\{0\}$ and $x\in N$ and consider the sequence $\widetilde z_n:=(x, np)$ in $N\times \mR^q$, as well as its image $z_n:=\pi(\widetilde z_n)$ in $(N\times \mR^q)/\Gamma$. Since $M:=(N\times \mR^q)/\Gamma$ is compact, there exists a convergent subsequence  of $(z_n)_n$, \emph{i.e.} there exists a strictly increasing sequence $(k_n)_n\subset\mN$ and $z_0\in M$ such that $z_{k_n}\to z_0$. Let $V$ be a neighbourhood of $z_0$, such that there exists $\widetilde V\subset N\times \mR^q$ with $\pi|_{\widetilde V}\colon \widetilde V\to V$ diffeomorphism. 

Let $\widetilde z_0:=(\pi|_{\widetilde V})^{-1}(z_0)$ and, for all $n\in \mN$ sufficiently large in order to ensure that $z_{k_n}\in V$, define $\widetilde y_n:=(\pi|_{\widetilde V})^{-1}(z_{k_n})$. Then the sequence $(\widetilde y_n)_n$ converges to $\widetilde z_0$. Since $z_{k_n}=\pi(\widetilde z_{k_n})=\pi(\widetilde y_n)$, there exists a sequence $(\gamma_n)_n\in\Gamma$, such that $\widetilde y_n=\gamma_n(\widetilde z_{k_n})$. We may now write according to the results obtained above that $\gamma_n(x,y)=(\gamma'_n(x), \gamma''_n(y))$, for all $(x,y)\in N\times \mR^q$, where $\gamma''_n(y)=B_{n} y+w_{n}$, with $B_{n}\in \CO(q)$ and $w_n\in\mR^q$, such that $[B_n, C]=0$ and $Cw_n=0$. The equality  $\widetilde y_n=\gamma_n(\widetilde z_{k_n})$ thus yields
$$\widetilde y_n=\gamma_n(\widetilde z_{k_n})=(\gamma'_n(x), \gamma''_n(k_np))=(\gamma'_n(x), k_n B_n p+w_n)\underset{n\to\infty}\longrightarrow \widetilde{z}_0.$$
Since $k_nB_np\in \mathrm{Im}(C)$ and $w_n\in\mathrm{Ker}(C)$, one can write $\widetilde{z}_0=(x_0, Cy_0+w_0)$ for some $y_0\in \mR^q$ and $w_0\in\mathrm{Ker}(C)$. 

Using that $\mathrm{Im}(C)\oplus \mathrm{Ker}(C)= \mR^q$, we deduce that $\gamma'_n(x)\to x_0$, $k_n B_n p\to C y_0$ and $w_n \to w_0$. In particular, $\gamma_n(x,0)=(\gamma'_n(x), w_n)\to (x_0, w_0)$. From this convergence and
the fact that $\Gamma$ acts properly discontinuously on $N\times \mR^q$, it follows that the sequence $(\gamma_n)_n$ is stationary, \emph{i.e.} there exists $n_0$ such that $\gamma_n=\gamma_{n_0}$, for all $n\geq n_0$.  In particular, $B_n=B_{n_0}$, for $n\geq n_0$, so from $k_n B_n p=k_n B_{n_0} p\to C y_0$, with $(k_n)_n\subset\mN$ strictly increasing, we conclude that  $B_{n_0} p=0$, so $p=0$, which contradicts the fact that $p\in\mathrm{Im}(C)\setminus\{0\}$. Thus we conclude that $X_2=0$.
\end{proof}

We are now ready for the second main result of this paper.

\begin{theorem}\label{thmholomorph}
	Let $(M,J,c)$ be a compact lcK manifold. If the K\"ahler cover $(\widetilde M, \ti J, g_K)$ is neither flat nor hyperk\"ahler, then every conformal vector field on $(M, c)$ is holomorphic.
\end{theorem}	

\begin{proof}
	Let $\xi$ be a conformal vector field on $(M,c)$. Then, according to Theorem~\ref{thmkilling}, $\xi$ is a Killing vector field with respect to the Gauduchon metric $g_0\in c$. 
	
	 If  the Lee form $\theta_0$ of $g_0$ vanishes identically, then $(M,g_0,J)$ is Kähler, and a standard argument shows that $\xi$ is holomorphic. Indeed, the Kähler form  $\Omega_0$ of $(g_0,J)$ is harmonic and so is its Lie derivative with respect to the Killing vector field $\xi$. On the other hand, since $\di\Omega_0=0$, Cartan's formula shows that $\mathcal{L}_\xi\Omega_0=\di(\xi\lrcorner\Omega_0)$ is also exact. A harmonic form which is exact vanishes identically, so $0=\mathcal{L}_\xi\Omega_0=g_0(\mathcal{L}_\xi J\cdot,\cdot),$ whence $\xi$ is holomorphic.

	 We thus assume for the rest of the proof that $\theta_0$ is not identically zero.

Let $\ti\xi$ denote the vector field induced by $\xi$ on $\widetilde M$, \emph{i.e.} $\pi_*(\ti\xi)=\xi$, where $\pi\colon \widetilde M\to M$ is the projection of the universal cover. By the last part of Theorem~\ref{thmkilling}, $\ti\xi$ is homothetic with respect to the K\"ahler metric~$g_K$. In particular, $\ti\xi$ is affine with respect to the Levi-Civita connection $\nabla^{g_K}$. We distinguish the following two cases:

	{\bf Case 1.} If $\mathrm{Hol}(g_K)$ is irreducible, then any transformation in the connected component of the identity of the group of affine transformations of $\nabla^{g_K}$ is holomorphic, \emph{cf.} \cite[Lemma 2.1.]{mo}, because $(\widetilde M, \ti J,  g_K)$ is assumed to be neither flat nor hyperk\"ahler. Applying this result to  the flow of $\ti\xi$, yields that $\ti\xi$ is a holomorphic vector field on $(\widetilde M, \ti J)$, which finishes the proof in the first case.
	 
{\bf Case 2.} If $\mathrm{Hol}(g_K)$ is reducible, then a result of M. Kourganoff, \cite[Theorem 1.5.]{k}, implies that the K\"ahler cover splits as a Riemannian product $(\widetilde M, g_K)\simeq (N, g_N)\times  (\mR^q,g_{\text{flat}})$, where $q$ is even and the metric $g_N$ is non-flat, incomplete and has irreducible holonomy. Applying Proposition~\ref{propN} to the action of $\Gamma:=\pi_1(M)$ on $(\widetilde M,  g_K)$, we conclude that $\ti \xi$ is tangent to $N$ and constant in the direction of $\mR^q$, \emph{i.e.} there exists a homothetic vector field $\zeta$ on $(N, g_N)$ $\ti\xi_{(x,y)}=\zeta_{x}$, for all $(x,y)\in N\times \mR^q$. 
	 
	 We argue by contradiction and assume that $\zeta$ is not holomorphic. Since $g_N$ has irreducible holonomy, we conclude, applying again Lemma 2.1. from \cite{mo}, that $(N,g_N)$ is hyperk\"ahler, so, in particular, $g_N$ is Ricci-flat. Thus also $(\widetilde M, g_K)$ is Ricci-flat and, consequently, the standard Weyl connection $D$ on $(M,c)$ is Weyl-Einstein. By a result of K.P.~Tod~\cite[Prop. 2.2]{tod} it follows that the dual vector field $T:=\theta_0^\sharp$ is  Killing with respect to $g_0$, which implies that $(M, J, g_0)$ is Vaisman, \emph{i.e.} $\nabla^{g_0} T=0$. Writing $g_K=e^{2\varphi}\tilde g_0$, with $\di\varphi=\ti\theta_0$, yields
	 $\mathcal{L}_{\widetilde T}g_K=2\ti g_0(\widetilde T,\widetilde T) g_K$.
	 
	Since $\ti g_0(\widetilde T,\widetilde T)$ is constant, the induced vector field $\widetilde T$ is homothetic on $(\widetilde M, g_K)$. By Proposition~\ref{propN} again, $\widetilde T$ is tangent to $N$ and is constant in the direction of $\mR^q$. Such a vector field can only be homothetic if it is Killing. Thus $\ti g_0(\widetilde T,\widetilde T)=0$, so $\theta_0=0$, which was excluded.
	 
	  Our assumption is thus false, showing that $\zeta$ is holomorphic on $N$, so $\ti\xi$ is holomorphic on $N\times \mR^q$, and therefore $\xi$ is a holomorphic vector field on~$(M, J)$.
\end{proof}

\begin{remark}
	An alternative argument for the second case in the proof of Theorem~\ref{thmholomorph} is the following. Assuming that $\xi$ is not holomorphic, and that $(\widetilde M,g_K)$ has reducible holonomy, we obtain as before that $(M, J, g_0)$ is Vaisman, so the universal cover $(\widetilde M, \ti g_0)$ carries a non-trivial parallel 1-form $\ti\theta_0$. Consequently, it splits as a Riemannian product $(\widetilde M, \ti g_0)=(\mR,\di\varphi^2)\times (S,g_S)$, with $\ti\theta_0=\di\varphi$ and $(S,g_S)$ complete. Then after the change of variable $r:=e^\varphi$, the Kähler metric on $\widetilde M$ reads $g_K=e^{2\varphi}\ti g_0=\di r^2+r^2g_S$. Thus $(\widetilde M,g_K)$ is isometric to the Riemannian cone of $(S,g_S)$, so it is irreducible by S. Gallot's result \cite[Prop. 3.1]{gallot}. This contradiction shows that $\xi$ is holomorphic.
\end{remark}

\end{document}